\documentclass[11pt]{amsart}

\usepackage{amssymb}
\usepackage{hyperref}
\usepackage{amsthm}
\usepackage[mathcal]{euscript}
\usepackage[T1]{fontenc}
\usepackage[latin1]{inputenc}
\usepackage{ae,aecompl}
\usepackage{enumitem}

\setenumerate[1]{label=$(\arabic*)$,leftmargin=*}
\setenumerate[2]{label=$(\alph*)$,leftmargin=*}
\setenumerate[3]{label=$(\roman*)$,leftmargin=*}
\setitemize[1]{label=$-$,leftmargin=*}
\setitemize[2]{label=$\circ$,leftmargin=*}
\setitemize[3]{label=$--$,leftmargin=*}

\newtheorem{Thm}{Theorem}[section]

\newtheorem{Cor}[Thm]{Corollary}

\numberwithin{equation}{section}

\def\pv#1{\ensuremath{{\sf#1}}}
\def\Cl#1{\ensuremath{\mathcal #1}}
\def\Om#1#2{\ensuremath{\overline\Omega_{#1}{\sf#2}}}

\def\omup#1#2#3{\ensuremath{\Omega^{#1}_{#2}{\sf#3}}}
\def\oms#1#2{\omup{\sigma}{#1}{#2}}

\def\omcpi#1#2{\omup{\kappa_{\pi}}{#1}{#2}}

\def\pj#1{\ensuremath{p_{\sf#1}}}

\def\malcev{\mathop{\raise1pt\hbox{\footnotesize$\bigcirc$\kern-8pt\raise1pt
      \hbox{\tiny$m$}\kern1pt}}}

\begin{document}

\author[J. Almeida]{J. Almeida}
\author[J. C. Costa]{J. C. Costa}
\author[M. Zeitoun]{M. Zeitoun}

\address[J. Almeida]{Centro de Matem\'atica e Departamento de
  Matem\'atica Pura, Faculdade de Ci\^encias, Universidade do Porto,
  Rua do Campo Alegre, 687, 4169-007 Porto, Portugal.
  \href{mailto:jalmeida@fc.up.pt}{jalmeida@fc.up.pt}}

\address[J. C. Costa]{Centro de Matem\'atica e Departamento de
  Matem\'atica e Aplica\c c\~oes, Universidade do Minho, Campus de Gualtar,
  4700-320 Braga, Portugal.
  {\href{mailto:jcosta@math.uminho.pt}{jcosta@math.uminho.pt}}}

\address[M. Zeitoun]{LaBRI, Universit\'e de Bordeaux \& CNRS UMR~5800.  351
  cours de la Lib\'eration, 33405 Talence Cedex, France.
{\href{mailto:mz@labri.fr}{mz@labri.fr}}}  

\title{Reducibility of pointlike problems}

\date{\today}

\begin{abstract}
  We show that the pointlike and the idempotent pointlike problems are
  reducible with respect to natural signatures in the following cases:
  the pseudovariety of all finite semigroups in which the order of
  every subgroup is a product of elements of a fixed set $\pi$ of
  primes; the pseudovariety of all finite semigroups in which every
  regular \Cl J-class is the product of a rectangular band by a group
  from a fixed pseudovariety of groups that is reducible for the
  pointlike problem, respectively graph reducible. Allowing only
  trivial groups, we obtain $\omega$-reducibility of the pointlike and
  idempotent pointlike problems, respectively for the pseudovarieties
  of all finite aperiodic semigroups (\pv A) and of all finite
  semigroups in which all regular elements are idempotents~(\pv{DA}).
\end{abstract}


\makeatletter

\@namedef{subjclassname@2010}{%
  \textup{2010} Mathematics Subject Classification}
\makeatother


\maketitle

\section{Introduction}
\label{sec:introduction}

For a pseudovariety \pv V of semigroups, the effective computation of
\pv V-pointlike subsets of finite semigroups intervenes in the
solution of various decision problems. One famous example is the
case of \pv G-pointlike sets, where \pv G is the pseudovariety of all
finite groups, for which a concrete algorithm was proposed by Henckell
and Rhodes \cite{Henckell&Rhodes:1991} and later proved by
Ash~\cite{Ash:1991}. It was conceived as a successful tool to
effectively decide whether a finite semigroup divides the power
semigroup of some finite group (see
\cite{Pin:1995,Henckell&Margolis&Pin&Rhodes:1991} for a history of the
problem). The computation of \pv V-idempotent pointlikes sets in turn
yields easily a solution of the membership problem for pseudovarieties
given by Mal'cev products of the form $\pv W\malcev\pv V$, provided
membership in \pv W is decidable \cite[Proposition
4.3]{Henckell:2004}.

The computation of pointlike and idempotent pointlike sets has been
carried out for several pseudovarieties. A general approach for
obtaining theoretical computability was devised by Steinberg and the
first author through tameness
\cite{Almeida&Steinberg:2000a,Almeida&Steinberg:2000b}. The idea is
that there is an obvious semi-algorithm to generate subsets of a given
finite semigroup that are not \pv V-pointlike (respectively not \pv
V-idempotent pointlike), provided \pv V is decidable. To generate the
favorable cases, one needs witnesses in the profinite semigroup,
namely pseudowords that evaluate to the given elements in the
semigroup and are equal over \pv V. The essential property for
tameness is \emph{reducibility}, which means that such witnesses may
be found among pseudowords of a restricted, effectively enumerable
form, such as $\omega$-words; more generally, terms in an implicit
signature should be taken. This general approach may and has been
considered for arbitrary finite systems of equations, the cases of
pointlike and idempotent pointlike sets corresponding respectively to
systems of the forms $x_1=\cdots=x_n$ and $x_1=\cdots=x_n=x_n^2$.

In this paper, we show that the reducibility property holds for
pointlike and idempotent pointlike problems for certain
pseudovarieties under simple assumptions. The two cases that we
consider here are the pseudovarieties of the form $\overline{\pv
  G_\pi}$, of all finite semigroups whose subgroups have orders whose
prime factors belong to the set $\pi$ of primes; and the
pseudovarieties of the form $\pv{DO}\cap\overline{\pv H}$ of all
finite semigroups in which all regular \Cl J-classes are products of
rectangular bands by groups from the pseudovariety \pv H. In
both instances, the case where all subgroups are trivial is of
interest and represents, respectively, the pseudovarieties \pv A, of
all finite aperiodic semigroups, and \pv{DA}, of all finite semigroups
in which all regular elements are idempotents.

\section{Preliminaries}
\label{sec:prelims}

The reader is referred to the standard bibliography
\cite{Almeida:1994a,Rhodes&Steinberg:2009qt} on finite semigroups for
undefined terminology and background.

By an \emph{implicit signature} we mean a set $\sigma$ of pseudowords
(also called implicit operations) over the pseudovariety \pv S of all
finite semigroups, that is, elements of some finitely generated free
profinite semigroup \Om AS, the only requirement being that binary
multiplication is one of them. Since all implicit signatures are
assumed to contain binary multiplication, we omit reference to it when
describing implicit signatures. By definition, pseudowords have a
natural interpretation in every finite (whence also in every
profinite) semigroup, so that every profinite semigroup has a natural
structure of $\sigma$-semigroup. The $\sigma$-subalgebra of~\Om AS
generated by $A$ is denoted \oms AS. Examples of implicit signatures
are given by $\omega$, consisting of the unary operation
$\omega$-power $\_^\omega$, and $\kappa$, consisting of the unary
operation $(\omega-1)$-power $\_^{\omega-1}$.

A subset $P$ of a finite semigroup $S$ is said to be \emph{pointlike} with
respect to a pseudovariety \pv V if, for every onto continuous homomorphism
$\varphi:\Om AS\to S$, there is a subset $P'$ of~\Om AS such that
$\varphi(P')=P$ and $\pj V(P')$ is a
singleton~\cite{Almeida:1996d}, where
  $\pj V:\Om AS\to \Om AV$ denotes the canonical projection from the free
  profinite semigroup to the relatively free pro-\pv V semigroup. An
equivalent formulation in terms of relational morphisms and further discussion
on the interest of this notion can be found in~\cite{Rhodes&Steinberg:2009qt}.
A simple and fruitful interpretation in terms of formal languages has
been formulated and established in~\cite{Almeida:1996d}. The set of all nonempty \pv
V-pointlike subsets of~$S$ constitutes a semigroup under set multiplication
and is denoted $\Cl P_\pv V(S)$. The semigroup of all nonempty subsets of~$S$
will be denoted by $\Cl P(S)$. The problem of computing \pv V-pointlikes is
said to be \emph{$\sigma$-reducible} for an implicit signature $\sigma$ if the
above set $P'$ can always be chosen to be a subset of~\oms AS.

As shown in~\cite{Almeida&Steinberg:2000a}, under suitable additional
assumptions, the $\kappa$-reducibility of the pointlike problem implies its
algorithmic solution. However, the resulting algorithm is merely theoretical
and completely impractical. In the case of the pseudovariety \pv A, a
structural algorithm for computing pointlike sets of finite semigroups has
been obtained by Henckell \cite{Henckell:1988}.
Generalizations and more transparent and shorter proofs can
  be found in~\cite{Henckell&Rhodes&Steinberg:2010} and in \cite{pzfo:2014}
  (the
  latter paper being based on the interpretation of pointlike sets
  from~\cite{Almeida:1996d}). Algorithms with the same flavor have been
obtained in~\cite{Almeida&Costa&Zeitoun:2006} for the pseudovarieties \pv J
and \pv R. Similar techniques to the ones developed
  in~\cite{pzfo:2014} have been applied to show that the pointlike sets of
  size 2 of the pseudovariety \pv{DA} are also effectively computable~\cite{PRZ:mfcs:13}.

Henckell's algorithm actually intervenes in the proof of $\kappa$-reducibility
for the \pv A-pointlike problem presented in Section~\ref{sec:A-k-red-for-ptl}
in the form of the extension obtained
in~\cite{Henckell&Rhodes&Steinberg:2010}. Also essential in our treatment of
idempotent pointlike sets in the aperiodic case is Henckell's result that
every \pv A-idempotent pointlike subset of a finite semigroup is contained in
some \pv A-pointlike set that is idempotent \cite{Henckell:2004}. The
extension of this result to pseudovarieties of the form $\pv A\malcev\pv V$ is
also attributed to Henckell
in~\cite[Theorem~4.5]{Henckell&Rhodes&Steinberg:2010b}.

\section{Reducibility of \texorpdfstring{\pv A}A-pointlike sets}
\label{sec:A-k-red-for-ptl}

Given a set $\pi$ of primes, its complement in the set of all primes
is denoted~$\pi'$. A \emph{$\pi$-group} is a finite group whose order
factors into primes from ~$\pi$. The pseudovariety of all finite
$\pi$-groups is denoted $\pv G_\pi$. For a pseudovariety \pv H of
groups, $\overline{\pv H}$ stands for the pseudovariety consisting of
all finite semigroups all of whose subgroups lie in~\pv H. Note that
$\overline{\pv G_\emptyset}=\pv A$.

Since the additive semigroup of positive integers $\mathbb{Z}_+$ is
free, its profinite completion $\widehat{\mathbb{Z}_+}$ is a free
profinite semigroup. For $\nu\in\widehat{\mathbb{Z}_+}$ and
an element $s$ of a profinite semigroup $S$, denote by $s^\nu$ the
image of $\nu$ under the unique continuous homomorphism
$\widehat{\mathbb{Z}_+}\to S$ that maps $1$ to~$s$.

In case $\pi=\emptyset$, we let $\nu_\pi=\omega+1$. Otherwise, let
$\nu_\pi$ be any accumulation point of the sequence $((p_1\cdots
p_n)^{n!})_n$ in~$\widehat{\mathbb{Z_+}}$, where $p_1,p_2,\ldots$ is
an enumeration of the set~$\pi$, possibly with repetitions. In case
$\pi$ consists of all primes, it is easy to see that $\nu_\pi=\omega$.
Note that $\overline{\pv G_\pi}$ is defined by the pseudoidentity
$x^{\nu_\pi}=x^\omega$ while $\pv G_{\pi'}$ is defined by
$x^{\nu_\pi-1}=1$. Denote by $\kappa_\pi$ the implicit signature 
obtained by enriching the signature $\kappa$ with the operations
$x^\mu$, if it is not already expressible by an $\omega$-term,
whenever there exists $k\in\mathbb{Z}_+$ such that
$k(\mu+1)=\nu_\pi-1$. In particular, one can easily check that
$\kappa$ coincides with both $\kappa_\emptyset$ and $\kappa_\pi$ if
$\pi$ consists of all primes.

A subset of~$\Cl P(S)$ is said to be \emph{downward closed} if,
whenever it contains $P$ and $Q\subseteq P$, it also contains $Q$. For
a subset $P$ of a finite semigroup $S$, we denote by $P^{\omega+*}$
the set $P^\omega\bigcup_{n\geq1}P^n$.
Given a set of primes $\pi$ and a semigroup $S$, we define %
$\Cl C\Cl P_\pi(S)$ to be the smallest downward closed subsemigroup
of~$\Cl P(S)$ containing all singleton subsets of~$S$ and which
contains $P^{\omega+*}$ whenever it contains an element $P$ which
generates a cyclic $\pi'$-subgroup of~$\Cl P(S)$.
It follows from \cite[Theorem~2.3]{Henckell&Rhodes&Steinberg:2010}
that the equality $\Cl C\Cl P_\pi(S)=\Cl P_{\overline{\pv G_\pi}}(S)$
holds for every finite semigroup $S$, which implies that %
$\Cl P_{\overline{\pv G_\pi}}(S)$~is computable in case $\pi$~is a
recursive set of primes. The following theorem strengthens the easy
direction of this result, showing not only that elements of $\Cl C\Cl
P_\pi(S)$ are $\overline{\pv G_\pi}$-pointlike subsets, but also that
this can be witnessed by $\kappa_\pi$-terms.
  
\begin{Thm}
  \label{t:HRS-kpi-reducible}
  Let $\pi$ be a set of primes and let $\varphi:\Om AS\to S$ be a
  continuous homomorphism onto a finite semigroup. Then every $P\in\Cl
  C\Cl P_\pi(S)$ has the following property:
  \begin{equation}
    \label{eq:prop-P}
    \parbox[c]{.8\textwidth}{there exists a function
      $\alpha_P:P\to\omcpi AS$ such that the equality
      $\varphi\,\alpha_P=\mathrm{id}_P$ holds
      and $p_{\overline{\pv G_\pi}}\,\alpha_P$ is constant.}
  \end{equation}
\end{Thm}

\begin{proof}
  To prove the theorem, we proceed by induction on the construction of
  the semigroup $\Cl C\Cl P_\pi(S)$ that is immediately derived from
  its definition. At the base of the induction, we take the subset of
  $\Cl P(S)$ consisting of the singleton subsets of~$S$, for which
  property~(\ref{eq:prop-P}) obviously holds, as the restriction of
  $\varphi$ to \omcpi AS is onto. For the induction steps, we need to
  distinguish three types of transformations on subsets of~$S$.
  
  For taking subsets, it suffices to observe that, if $Q\subseteq P$,
  $\alpha_P$ verifies~\eqref{eq:prop-P}, and $\alpha_Q$ is taken to be
  the restriction of $\alpha_P$ to $Q$, then $\alpha_Q$
  verifies~\eqref{eq:prop-P} for $Q$ in the place of $P$.

  For taking products, suppose that $P,Q\in\Cl P(S)$ are such that
  $\alpha_P$ and $\alpha_Q$ verify the corresponding
  properties~\eqref{eq:prop-P} for $P$ and $Q$. Let $R=PQ$ and define
  $\alpha_R:R\to\omcpi AS$ by letting, for $r\in R$,
  $\alpha_R(r)=\alpha_P(p)\alpha_Q(q)$ where $r=pq$ is any chosen
  factorization of $r$ with $p\in P$ and $q\in Q$. Given $r\in R$,
  consider its chosen factorization $r=pq$ with $p\in P$ and $q\in Q$.
  Then we have
  $$\varphi(\alpha_R(r)) %
  =\varphi(\alpha_P(p)\alpha_Q(q)) %
  =\varphi(\alpha_P(p))\varphi(\alpha_Q(q)) %
  =pq %
  =r.$$
  Similarly, one shows that $p_{\overline{\pv G_\pi}}\,\alpha_R$ is a
  constant mapping. This shows that $R$ also has
  property~\eqref{eq:prop-P}.

  Finally, suppose that $P$ is a subset of~$S$ for which there is a
  function $\alpha_P$ satisfying property~\eqref{eq:prop-P} and such
  that $P$ generates a cyclic $\pi'$-subgroup of~$\Cl P(S)$. Let
  $Q=P^{\omega+*}$ and note that $Q=\bigcup_{n=1}^mP^n$ for some $m$.

  By the assumption that the cyclic subgroup of $\Cl P(S)$ generated by $P$ is
  a $\pi'$-group, the equality $P^{\nu_\pi}=P$ holds. Hence, there is some
  $\mu\in\widehat{\mathbb{Z}_+}$ such that
  $\nu_\pi-1=k(\mu+1)$, where $k$ denotes the order of the
    $\pi'$-group generated by $P$. This implies that the unary operation
  $x^\mu$ belongs to~$\kappa_\pi$.

  Note that $P^k$ is a subsemigroup of~$S$. Given $p\in P$, since
  $p\in P^{k\ell+1}$ for every $\ell\ge 0$, there exists a
  factorization $p=q_1\cdots q_\ell p'$, where each $q_i\in P^k$ and
  $p'\in P$. Choosing $\ell=|P^k|+1$, there are integers $i$ and $j$
  such that $1\le i<j\le\ell$ such that $q_1\cdots q_i=q_1\cdots q_j$.
  Hence, there are factorizations
  $$p %
  =q_1\cdots q_i\cdot(q_{i+1}\cdots q_j)^\omega %
  \cdot q_{j+1}\cdots q_\ell p' %
  =s\cdot e\cdot t,$$ %
  where $s=s_1\cdots s_k$, $e=e^2=e_1\cdots e_k$, and $t$ together
  with all $s_m$ and all $e_m$ belong to~$P$. Having chosen such a
  factorization, we now define
  $$\beta(p)=\alpha_P(s_1)\cdots\alpha_P(s_k) %
  \bigl(\alpha_P(e_1)\cdots\alpha_P(e_k)\bigr)^\mu %
  \alpha_P(t).$$ %
  Note that $\beta(p)$ is given by a $\kappa_\pi$-term. Moreover, we
  obtain the equalities
  \begin{align*}
    \varphi(\beta(p))&=s_1\cdots s_k(e_1\cdots e_k)^\mu t=se^\mu t=set=p, \\
    \pj{\overline{\pv G_\pi}}(\beta(p))
    &=\pj{\overline{\pv G_\pi}}(\alpha_P(p))^{k+\mu k+1}
    =\pj{\overline{\pv G_\pi}}(\alpha_P(p))^{\nu_\pi}
    =\pj{\overline{\pv G_\pi}}(\alpha_P(p))^\omega.
  \end{align*}
  It follows that we may assume that the function $\alpha_P$ is such
  that the constant value of $p_{\overline{\pv G_\pi}}\,\alpha_P$ is
  an idempotent $f$. This entails that, for the function
  $\alpha_{P^n}$ defined on $P^n$ according to the product case of the
  induction, the composite $p_{\overline{\pv G_\pi}}\,\alpha_{P^n}$
  still has the same constant value $f$. Hence, we may extract from
  the relation $\bigcup_{n=1}^m\alpha_{P^n}$ a function
  $\alpha_{P^{\omega+*}}$ which has the required property
  \eqref{eq:prop-P} for $P^{\omega+*}$. This concludes the inductive
  step and the proof of the theorem.
\end{proof}

\begin{Cor}
  \label{c:ovGpi-reducible}
  If $\pi$ is an arbitrary set of primes, then the pseudovariety
  $\overline{\pv G_\pi}$ is $\kappa_\pi$-reducible for pointlike sets
  and also $\kappa_\pi$-reducible for idempotent pointlike sets. In
  particular, \pv A~is $\kappa$-reducible for both pointlike sets and
  idempotent pointlike sets.
\end{Cor}

\begin{proof}
  Recall that by~\cite[Theorem~2.3]{Henckell&Rhodes&Steinberg:2010},
  the equality $\Cl C\Cl P_\pi(S)=\Cl P_{\overline{\pv G_\pi}}(S)$
  holds for every finite semigroup $S$. The result for pointlike sets
  then follows directly from
  Theorem~\ref{t:HRS-kpi-reducible}. For idempotent pointlike sets, it
  suffices to invoke, additionally,
  \cite[Theorem~4.5]{Henckell&Rhodes&Steinberg:2010b} since $\pv
  A\malcev\overline{\pv G_\pi}=\overline{\pv G_\pi}$.
\end{proof}

\section{Reducibility of \texorpdfstring{\pv{DA}}{DA}-pointlike sets}
\label{sec:DA-reduc}

For the pseudovariety \pv{DA}, the proof below of
$\kappa$-reducibility of the pointlike problem is inspired by
\cite[Lemma~5.10]{Almeida&Costa&Zeitoun:2004}. It is of a more
syntactical nature than the proof of
Theorem~\ref{t:HRS-kpi-reducible}, using central basic factorizations
as introduced in~\cite{Almeida:1996c}. A further parameter that we
proceed to introduce plays a key role. Given $u\in\Om AS$, there may or
may not be a central basic factorization of the form %
$u=u_0a_0\cdot u'_0\cdot b_0u''_0$ with
$c(u)=c(u_0)\uplus\{a_0\}=c(u''_0)\uplus\{b_0\}$. We are interested in
iterating such a factorization on the middle factor $u'_0$ while it is
possible to do so without reducing the content, that is, while
$c(u'_0)=c(u)$. The supremum of the number of times we can keep
iterating such a factorization on the middle factor is denoted
$\|u\|$.

\begin{Thm}
  \label{t:DA-k-red}
  Let \pv H be a pseudovariety of groups and suppose that
  $\sigma_0$~is an implicit signature such that the \pv H-pointlike
  problem is $\sigma_0$-reducible. Let $\pv V=\pv{DO}\cap\overline{\pv
    H}$ and let $\sigma=\sigma_0\cup\{\_^\omega\}$. Suppose $S$ is a
  finite semigroup, $\varphi:\Om AS\to S$ is an onto homomorphism, and
  $\{s_1,\ldots,s_n\}$ is a \pv{V}-pointlike subset of~$S$. Given
  $u_1,\ldots,u_n\in\Om AS$ such that $\varphi(u_i)=s_i$
  ($i=1,\ldots,n$) and
  \begin{equation}
    \label{eq:1}
    \pj{V}(u_1)=\cdots=\pj{V}(u_n),
  \end{equation}
  there exist $w_1,\ldots,w_n\in\oms AS$
  such that
  \begin{equation}
    \label{eq:2}
    \varphi(w_i)=s_i\ (i=1,\ldots,n) %
    \mbox{ and } %
    \pj{V}(w_1)=\cdots=\pj{V}(w_n).
  \end{equation}
\end{Thm}

For shortness, we say that the $n$-tuple $(w_1,\ldots,w_n)$ of
$\sigma$-terms is a \emph{$(\pv V,\sigma)$-reduction} of
$(u_1,\ldots,u_n)$ if it satisfies property~\eqref{eq:2}.

\begin{proof}
  Without loss of generality, we may assume that $S$ has a content
  function, that is that the content function $c:\Om
  AS\to\mathcal{P}(A)$ factors through~$\varphi$. Note that,
  by~(\ref{eq:1}), we must have $c(u_1)=\cdots=c(u_n)$.
  
  We show by induction on $|c(u_1)|$ that the $u_i$ may be replaced by
  $\kappa$-terms $w_i$ satisfying properties~\eqref{eq:2}. The case
  $c(u_1)=\{a\}$ is rather easy. Indeed, in case $u_1$ is a finite
  word, by~(\ref{eq:1}) so are all $u_i$ and, therefore, they are
  $\sigma$-terms. Otherwise, for each $i$, we have $u_i=a^\omega u_i$.
  As~\eqref{eq:1} entails $\pj H(u_1)=\cdots=\pj H(u_n)$, there exists
  an $(\pv H,\sigma_0)$-reduction $(w_1',\ldots,w_n')$ of
  $(u_1,\ldots,u_n)$. Then, $w_i=a^\omega w_i'$ ($i=1,\ldots,n$)
  defines a $(\pv V,\sigma)$-reduction of~$(u_1,\ldots,u_n)$.

  Suppose that the claim holds whenever $|c(u_1)|<K$ and suppose an
  instance of the problem is given in which $|c(u_i)|=K$. Factorize
  $u_i$ as
  \begin{equation}
    \label{eq:4}
    u_i=%
    u_{i,0}a_0u_{i,1}a_1\cdots u_{i,l}a_l%
    \cdot u'_{i,l}\cdot %
    b_lu''_{i,l}\cdots b_1u''_{i,1}b_0u''_{i,0} %
  \end{equation}
  where %
  $c(u_i) %
  =c(u_{i,p})\uplus\{a_p\} %
  =c(u''_{i,q})\uplus\{b_q\}$ %
  ($p,q=0,\ldots,l$). We take here $l$ to be arbitrary if
  $\|u_i\|=\infty$ or $l=\|u_i\|$ otherwise. By~\eqref{eq:1} and
    using~\cite{Almeida:1996c}, we deduce that
  the $\|u_i\|$ are the same for all~$i$ and, moreover, the sequences
  of markers $a_0,a_1,\ldots,a_l$ and $b_0,b_1,\ldots,b_l$ are also
  the same for all~$i$ and \pv{V} satisfies each of the following
  pseudoidentities for all $i,j\in\{1,\ldots,n\}$ and
  $p\in\{0,\ldots,l\}$:
  $$u_{i,p}=u_{j,p},\ %
  u'_{i,l}=u'_{j,l},\ %
  u''_{i,p}=u''_{j,p}$$
  By construction, $|c(u_{i,p})|=|c(u''_{i,p})|=|c(u_1)|-1$ and so, by
  the induction hypothesis, for each $p$ there exist %
  $(\pv V,\sigma)$-reductions $(w_{1,p},\ldots,w_{n,p})$ of
  $(u_{1,p},\ldots,u_{n,p})$ and $(w''_{1,p},\ldots,w''_{n,p})$ of
  $(u''_{1,p},\ldots,u''_{n,p})$.
  
  We next distinguish two cases. In the first case, we assume
  $\|u_1\|<\infty$. By the choice of $l=\|u_i\|=\|u_1\|$, either
  $|c(u'_{i,l})|<|c(u_1)|$ or there is a factorization of $u'_{i,l}$
  of one of the forms $\alpha x\beta$ or $\alpha y\beta x\gamma$ with
  $c(u'_{i,l})=c(\alpha)\uplus\{x\}=c(\beta)\uplus\{x\}$ or
  $c(u'_{i,l})=c(\alpha y\beta)\uplus\{x\}=c(\beta
  x\gamma)\uplus\{y\}$, respectively. Moreover, the same case occurs
  for all $i\in\{1,\ldots,n\}$ and the factors in the same positions
  must have the same value under the projection \pj{V}. Applying the
  induction hypothesis again to each of the factors of the $u'_{i,l}$
  thus determined, we deduce that there exists a %
  $(\pv V,\sigma)$-reduction $(w'_{1,l},\ldots,w'_{n,l})$ of
  $(u'_{1,l},\ldots,u'_{n,l})$. One may then verify that, taking
  $$w_i = %
  w_{i,0}a_0 w_{i,1}a_1\cdots w_{i,l}a_l %
  \cdot w'_{i,l}\cdot %
  b_lw''_{i,l}\cdots b_1w''_{i,1} b_0w''_{i,0}$$
  defines a $(\pv V,\sigma)$-reduction $(w_1,\ldots,w_n)$ of the
  original $n$-tuple $(u_1,\ldots,u_n)$.
  
  It remains to handle the case where $\|u_1\|=\infty$. Consider, for
  each $l$, the $n$-tuple of pairs of~$S\times S$
  $$\bigl(\varphi(w_{i,0}a_0 w_{i,1}a_1\cdots w_{i,l}a_l), %
  \varphi(b_lw''_{i,l}\cdots b_1w''_{i,1} b_0w''_{i,0})\bigr) %
  \quad(i=1,\ldots,n).$$
  Since $S$ is finite, there are indices $k$ and $l$ such that $k<l$
  and the $n$-tuples corresponding to these two indices coincide.
  Thus, for every $i=1,\ldots,n$, we have
  \begin{align*}
    \lefteqn{\varphi(w_{i,0}a_0 w_{i,1}a_1\cdots w_{i,k}a_k)}\\
    &\quad=
    \varphi(w_{i,0}a_0 w_{i,1}a_1\cdots w_{i,k}a_k
    (w_{i,k+1}a_{k+1}\cdots w_{i,l}a_l))\\
    &\quad=
    \varphi(w_{i,0}a_0 w_{i,1}a_1\cdots w_{i,k}a_k
    (w_{i,k+1}a_{k+1}\cdots w_{i,l}a_l)^\omega),
  \end{align*}
  and, similarly,
  \begin{align*}
    \lefteqn{\varphi(b_kw''_{i,k}\cdots b_1w''_{i,1} b_0w''_{i,0})}\\
    &\quad=
    \varphi((b_lw''_{i,l}\cdots b_{k+1}w''_{i,k+1})^\omega
    b_kw''_{i,k}\cdots b_1w''_{i,1} b_0w''_{i,0}).
  \end{align*}
  Let $(w'_{1,k},\ldots,w'_{n,k})$ be an $(\pv H,\sigma)$-reduction of
  $(u'_{1,k},\ldots,u'_{n,k})$. Note that since $S$ is assumed
    to have  a content function, the content of
    $w'_{i,k}$ is the same as that of $u'_{i,k}$. Take
  \begin{align*}
    \lefteqn{w_i=  %
      w_{i,0}a_0 w_{i,1}a_1\cdots w_{i,k}a_k %
      (w_{i,k+1}a_{k+1}\cdots w_{i,l}a_l)^\omega} \\
    &\qquad\cdot w'_{i,k} %
    \cdot (b_lw''_{i,l}\cdots b_{k+1}w''_{i,k+1})^\omega %
    b_kw''_{i,k}\cdots b_1w''_{i,1} b_0w''_{i,0}.
  \end{align*}
  Then again one verifies that $(w_1,\ldots,w_n)$ is a $(\pv
  V,\sigma)$-reduction of the original $n$-tuple $(u_1,\ldots,u_n)$.
\end{proof}

\begin{Cor}
  \label{c:DOcapbarH-ptl}
  If the pseudovariety of groups \pv H is $\sigma$-reducible for the
  pointlike problem, then $\pv V=\pv{DO}\cap\overline{\pv H}$ is
  $\sigma\cup\{\_^\omega\}$-reducible for the pointlike problem.
\end{Cor}

\begin{proof}
  Let $S$ be a finite semigroup and let $\{s_1,\ldots,s_n\}$ be a
  \pv{V}-pointlike subset. Fix an onto homomorphism $\varphi:\Om
  AS\to S$. By a general compactness result \cite{Almeida:1996d} there
  are $u_1,\ldots,u_n\in\Om AS$ such that $\varphi (u_i)=s_i$
  ($i=1,\ldots,n$) and $\pj{V}(u_1)=\cdots=\pj{V}(u_n)$. The result
  now follows immediately from Theorem~\ref{t:DA-k-red}.
\end{proof}

The same approach allows us to deal with idempotent pointlike sets.
However, a stronger assumption is needed on the pseudovariety of
groups \pv H. Recall that a system of equations may be associated with
a directed graph $\Gamma$ by viewing both vertices and edges as
variables and assigning to each edge $x\xrightarrow yz$ the equation
$xy=z$. For such a system, we may consider \emph{constraints} in a
finite semigroup, given by a function $\psi:\Gamma\to S$. Given a
continuous onto homomorphism $\varphi:\Om AS\to S$, a \pv
V-\emph{solution} of a thus constrained system is a function
$\gamma:\Gamma\to\Om AS$ such that $\varphi(\gamma(x))=\psi(x)$ for
every $x\in\Gamma$ and the pseudoidentity
$\gamma(x)\gamma(y)=\gamma(z)$ holds in~\pv V for every edge
$x\xrightarrow yz$. A pseudovariety \pv V is said to be
\emph{$\sigma$-reducible for systems of graph equations}, or
\emph{graph $\sigma$-reducible} for shortness, if every constrained
system of equations associated with a finite directed graph $\Gamma$
that admits a \pv V-solution also admits a \pv V-solution
$\gamma:\Gamma\to\oms AS$.

\begin{Thm}
  \label{t:DOcapbarH-iptl}
  Let \pv H be a pseudovariety of groups that is graph
  $\sigma_0$-reducible. Let $\pv V=\pv{DO}\cap\overline{\pv H}$ and
  let $\sigma=\sigma_0\cup\{\_^\omega\}$. Suppose $S$ is a finite
  semigroup, $\varphi:\Om AS\to S$ is an onto homomorphism, and
  $\{s_1,\ldots,s_n\}$ is a \pv{V}-idempotent pointlike subset of~$S$.
  Given $u_1,\ldots,u_n\in\Om AS$ such that $\varphi(u_i)=s_i$
  ($i=1,\ldots,n$) and
  \begin{equation*}
    \label{eq:1id}
    \pj{V}(u_1)=\cdots=\pj{V}(u_n)=\pj{V}(u_n^2),
  \end{equation*}
  there exist $w_1,\ldots,w_n\in\oms AS$
  such that
  \begin{equation*}
    \label{eq:2id}
    \varphi(w_i)=s_i\ (i=1,\ldots,n) %
    \mbox{ and } %
    \pj{V}(w_1)=\cdots=\pj{V}(w_n)=\pj{V}(w_n^2).
  \end{equation*}
\end{Thm}

\begin{proof}
  The proof follows the same lines as that of
  Theorem~\ref{t:DA-k-red}. We only mention in detail the necessary
  adaptations. First, the hypothesis that the $\pj{V}(u_i)$ are (the
  same) idempotent implies that $\|u_i\|=\infty$, which restricts the
  type of cases that need to be considered in the main induction step.
  However, the induction argument does not reduce the idempotent
  pointlike problem to the same problem on smaller content, but
  rather to the pointlike problem, which has already been treated in
  Theorem~\ref{t:DA-k-red}. This is why we need to assume again that
  the \pv H-pointlike problem is $\sigma_0$-reducible.

  The other point where a modification is needed is when handling the
  construction of the $\sigma$-terms $w_{i,k}'$. Since the singleton
  subsets $\{a\}$ of~$A^*$ are recognizable, by replacing $S$
  by a suitable finite semigroup, we may assume that
  $\varphi^{-1}(\varphi(a))=\{a\}$ for every $a\in A$. By assumption,
  we know that, for each $i\in\{1,\ldots,n\}$, the following
  pseudoidentity holds in~\pv H:
  \begin{equation*}
    u_{i,0}a_0\cdots u_{i,k}a_k
    \cdot u'_{i,k}\cdot
    b_ku''_{i,k}\cdots b_0u''_{i,0}=1.
  \end{equation*}
  Consider the directed graph $\Gamma$ with $n$ cycles of length
  $4k+5$ based at the same vertex, where the $i$th cycle has
  successive edges
  \begin{equation*}
    x_{i,0},y_{i,0},\ldots,x_{i,k},y_{i,k},z_i,
    y''_{i,k},x''_{i,k},\ldots,y''_{i,0},x''_{i,0}.
  \end{equation*}
  The corresponding system of equations is constrained as follows. The
  edge constraints are given by:
  \begin{align*}
    \psi(x_{i,j})&=\varphi(u_{i,j}),\
    \psi(y_{i,j})=\varphi(a_j),\\
    \psi(x''_{i,j})&=\varphi(u''_{i,j}),\
    \psi(y''_{i,j})=\varphi(b_j),\\
    \psi(z_i)&=\varphi(u'_{i,k}).
  \end{align*}
  For the vertex constraints, we may take an arbitrary constraint $s$
  at the common vertex $v_0$ of the $n$ cycles and then take for the
  constraint of any other vertex $v$ the product of $s$ by the
  constraints of successive edges of the unique path leading from $v_0$
  to~$v$.

  An \pv H-solution of the above system is obtained by assigning to
  the edge variables the values given by
  \begin{align*}
    x_{i,j}&\mapsto u_{i,j},\
    y_{i,j}\mapsto a_j,\ 
    x''_{i,j}\mapsto u''_{i,j},\
    y''_{i,j}\mapsto b_j,\ 
    z_i\mapsto u'_{i,k},
  \end{align*}
  as well as adequate values to the vertex variables. Since \pv H is
  graph $\sigma_0$-reducible by hypothesis, the above constrained
  system admits a solution $\gamma:\Gamma\to\oms AS$. We use $\gamma$
  to define
  \begin{equation*}
    \bar w_{i,j}=\gamma(x_{i,j}),\
    \bar w''_{i,j}=\gamma(x''_{i,j}),\
    w'_{i,k}=\gamma(z_i).
  \end{equation*}
  Finally, we let
  \begin{align*}
    \lefteqn{w_i=  %
      \bar w_{i,0}a_0 \bar w_{i,1}a_1\cdots \bar w_{i,k}a_k %
      (w_{i,k+1}a_{k+1}\cdots w_{i,l}a_l)^\omega} \\
    &\qquad\cdot \bar w'_{i,k} %
    \cdot (b_lw''_{i,l}\cdots b_{k+1}w''_{i,k+1})^\omega %
    b_k\bar w''_{i,k}\cdots b_1\bar w''_{i,1} b_0\bar w''_{i,0}.
  \end{align*}
  Then, each $\pj{V}(w_i)$ is a group element and it must in fact be
  an idempotent because $\gamma$ is an \pv H-solution of the system
  determined by the graph $\Gamma$. Moreover,
  $\varphi(w_i)=\varphi(u_i)$ because of the choice of the pair $k,l$
  and of the constraints on the edges of~$\Gamma$. Hence \pv V is
  $\sigma$-reducible for idempotent pointlike sets.
\end{proof}

\begin{Cor}
  \label{c:DOcapbarH-iptl}
  If the pseudovariety of groups \pv H is graph $\sigma$-reducible,
  then $\pv V=\pv{DO}\cap\overline{\pv H}$ is
  $\sigma\cup\{\_^\omega\}$-reducible for the idempotent pointlike
  problem.\qed
\end{Cor}

The case of \pv{DA} has deserved the most interest among the
pseudovarieties of the form $\pv{DO}\cap\overline{\pv H}$.

\begin{Cor}
  \label{c:DA-ptl}
  The pseudovariety \pv{DA} is $\omega$-reducible both for the
  pointlike and idempotent pointlike problems.\qed
\end{Cor}

It is natural to expect that the essential ingredients in the proof of
complete $\kappa$-tameness of the pseudovariety \pv R should apply
to~\pv{DA}. Yet the highly technical proof
in~\cite{Almeida&Costa&Zeitoun:2005b} remains to be adapted as only
part of such a program has been carried out
\cite{Moura:2009a,Moura:2009b}.

Further examples of pseudovarieties for which one may apply
Corollaries~\ref{c:DOcapbarH-ptl} and~\ref{c:DOcapbarH-iptl} are
recorded in the following corollaries.

\begin{Cor}
  \label{c:DO}
  The pseudovariety \pv{DO} is $\kappa$-reducible for both pointlike
  and idempotent pointlike sets.
\end{Cor}

\begin{proof}
  To apply Corollaries~\ref{c:DOcapbarH-ptl}
  and~\ref{c:DOcapbarH-iptl}, one just needs to observe that
  $\pv{DO}\subseteq\pv S=\overline{\pv G}$ and \pv G is graph
  $\kappa$-reducible by~\cite[Theorem~4.9]{Almeida&Steinberg:2000a},
  which depends on Ash's seminal results~\cite{Ash:1991}.
\end{proof}

Denote by \pv{Ab} the pseudovariety of all finite Abelian groups.

\begin{Cor}
  \label{c:DOcapbarAb}
  The pseudovariety $\pv{DO}\cap\overline{\pv{Ab}}$ is
  $\kappa$-reducible for both pointlike and idempotent pointlike sets.
\end{Cor}

\begin{proof}
  The proof is similar to that of Corollary~\ref{c:DO}, taking into
  account that in fact the pseudovariety \pv{Ab} is $\kappa$-reducible
  for arbitrary finite systems of equations
  \cite{Almeida&Delgado:2001} and, thus, in particular, it is graph
  $\kappa$-reducible.
\end{proof}

For a prime $p$, denote by $\pv G_p$ the pseudovariety of all finite
$p$-groups. While $\pv G_p$ is not graph $\kappa$-reducible, as was
observed in~\cite{Almeida&Steinberg:2000a}, the first author has
constructed a signature $\sigma$ containing $\kappa$ and such that
$\pv G_p$ is graph $\sigma$-reducible~\cite{Almeida:1999c}. The proof
of this result depends on an earlier weaker result of Steinberg
\cite{Steinberg:1998a}.

\begin{Cor}
  \label{c:DOcapbarGp}
  The pseudovariety $\pv{DO}\cap\overline{\pv{G}_p}$ is
  $\sigma$-reducible for both pointlike and idempotent pointlike sets,
  where $\sigma$ is the implicit signature constructed
  in~\cite{Almeida:1999c}.\qed
\end{Cor}

\section*{Acknowledgments}

Work partly supported by the \textsc{Pessoa} French-Portuguese project
``Separation in automata theory: algebraic, logical, and combinatorial
aspects''.
The work of the first author was also partially supported by CMUP
\linebreak (UID/MAT/00144/2013), which is funded by FCT (Portugal)
with national (MEC) and European structural funds through the programs
FEDER, under the partnership agreement PT2020.
The work of the second author was also partially supported by the
European Regional Development Fund, through the program COMPETE, and
by the Portuguese Government through FCT under the project
PEst-OE/MAT/UI0013/2014.
The work of the third author was partly supported by ANR 2010 BLAN
0202 01 FREC.

\bibliographystyle{amsplain}

\begin{thebibliography}{10}

\bibitem{Almeida:1994a}
J.~Almeida, \emph{Finite semigroups and universal algebra}, World Scientific,
  Singapore, 1995.

\bibitem{Almeida:1996c}
\bysame, \emph{A syntactical proof of locality of {DA}}, Int. J. Algebra
  Comput. \textbf{6} (1996), 165--177.

\bibitem{Almeida:1996d}
\bysame, \emph{Some algorithmic problems for pseudovarieties}, Publ. Math.
  Debrecen \textbf{54 Suppl.} (1999), 531--552.

\bibitem{Almeida:1999c}
\bysame, \emph{Dynamics of implicit operations and tameness of pseudovarieties
  of groups}, Trans. Amer. Math. Soc. \textbf{354} (2002), 387--411.

\bibitem{Almeida&Costa&Zeitoun:2004}
J.~Almeida, J.~C. Costa, and M.~Zeitoun, \emph{Tameness of pseudovariety joins
  involving {R}}, Monatsh. Math. \textbf{146} (2005), 89--111.

\bibitem{Almeida&Costa&Zeitoun:2005b}
\bysame, \emph{Complete reducibility of systems of equations with respect to
  \pv{R}}, Portugal. Math. \textbf{64} (2007), 445--508.

\bibitem{Almeida&Costa&Zeitoun:2006}
\bysame, \emph{Pointlike sets with respect to {R} and {J}}, J. Pure Appl.
  Algebra \textbf{212} (2008), 486--499.

\bibitem{Almeida&Delgado:2001}
J.~Almeida and M.~Delgado, \emph{Sur certains syst\`emes d'\'equations avec
  contraintes dans un groupe libre---addenda}, Portugal. Math. \textbf{58}
  (2001), no.~4, 379--387.

\bibitem{Almeida&Steinberg:2000a}
J.~Almeida and B.~Steinberg, \emph{On the decidability of iterated semidirect
  products and applications to complexity}, Proc. London Math. Soc. \textbf{80}
  (2000), 50--74.

\bibitem{Almeida&Steinberg:2000b}
\bysame, \emph{Syntactic and global semigroup theory, a synthesis approach},
  Algorithmic Problems in Groups and Semigroups (J.~C. Birget, S.~W. Margolis,
  J.~Meakin, and M.~V. Sapir, eds.), Birkh{\"a}user, 2000, pp.~1--23.

\bibitem{Ash:1991}
C.~J. Ash, \emph{Inevitable graphs: a proof of the type {II} conjecture and
  some related decision procedures}, Int. J. Algebra Comput. \textbf{1} (1991),
  127--146.

\bibitem{Henckell:1988}
K.~Henckell, \emph{Pointlike sets: the finest aperiodic cover of a finite
  semigroup}, J. Pure Appl. Algebra \textbf{55} (1988), no.~1-2, 85--126.

\bibitem{Henckell:2004}
\bysame, \emph{Idempotent pointlike sets}, Int. J. Algebra Comput. \textbf{14}
  (2004), 703--717, International Conference on Semigroups and Groups in honor
  of the 65th birthday of Prof. John Rhodes.

\bibitem{Henckell&Margolis&Pin&Rhodes:1991}
K.~Henckell, S.~W. Margolis, {J.-\'E.} Pin, and J.~Rhodes, \emph{Ash's type
  {${\rm II}$} theorem, profinite topology and {M}al'cev products. {I}}, Int.
  J. Algebra Comput. \textbf{1} (1991), no.~4, 411--436.

\bibitem{Henckell&Rhodes:1991}
K.~Henckell and J.~Rhodes, \emph{The {t}heorem of {K}nast, the {PG}={BG} and
  {T}ype {II} {C}onjectures}, Monoids and Semigroups with Applications
  (Singapore) (J.~Rhodes, ed.), World Scientific, 1991, pp.~453--463.

\bibitem{Henckell&Rhodes&Steinberg:2010}
K.~Henckell, J.~Rhodes, and B.~Steinberg, \emph{Aperiodic pointlikes and
  beyond}, Int. J. Algebra Comput. \textbf{20} (2010), no.~2, 287--305.

\bibitem{Henckell&Rhodes&Steinberg:2010b}
K.~Henckell, J.~Rhodes, and B.~Steinberg, \emph{A profinite approach to stable
  pairs}, Int. J. Algebra Comput. \textbf{20} (2010), 269--285.

\bibitem{Moura:2009a}
A.~Moura, \emph{Representations of the free profinite object over {DA}}, Int.
  J. Algebra Comput. (2011), 675--701.

\bibitem{Moura:2009b}
\bysame, \emph{The word problem for $\omega$-terms over {DA}}, Theor. Comp.
  Sci. (2011), 6556--6569.

\bibitem{Pin:1995}
J.-E. Pin, \emph{{BG=PG}: A success story}, Semigroups, Formal Languages and
  Groups (Dordrecht) (J.~Fountain, ed.), vol. 466, Kluwer, 1995, pp.~33--47.

\bibitem{PRZ:mfcs:13}
T.~Place, L.~van Rooijen, and M.~Zeitoun, \emph{Separating regular languages by
  piecewise testable and unambiguous languages}, {MFCS'13}, Lect. Notes Comp.
  Sci., vol. 8087, Springer, 2013, pp.~729--740.

\bibitem{pzfo:2014}
Thomas Place and Marc Zeitoun, \emph{Separating regular languages with
  first-order logic}, Proceedings of the Joint Meeting of the 23rd EACSL Annual
  Conference on Computer Science Logic (CSL'14) and the 29th Annual ACM/IEEE
  Symposium on Logic in Computer Science (LICS'14) (New York, NY, USA), ACM,
  2014, pp.~75:1--75:10.

\bibitem{Rhodes&Steinberg:2009qt}
J.~Rhodes and B.~Steinberg, \emph{The $q$-theory of finite semigroups},
  Springer Monographs in Mathematics, Springer, 2009.

\bibitem{Steinberg:1998a}
B.~Steinberg, \emph{Inevitable graphs and profinite topologies: some solutions
  to algorithmic problems in monoid and automata theory, stemming from group
  theory}, Int. J. Algebra Comput. \textbf{11} (2001), 25--71.

\end{thebibliography}

\providecommand{\bysame}{\leavevmode\hbox to3em{\hrulefill}\thinspace}
\providecommand{\MR}{\relax\ifhmode\unskip\space\fi MR }
\providecommand{\MRhref}[2]{%
  \href{http://www.ams.org/mathscinet-getitem?mr=#1}{#2}
}
\providecommand{\href}[2]{#2}

\end{document}